\documentclass[12pt,notitlepage]{amsart}
\usepackage{amsbsy,amsmath,amsthm,amssymb,color,verbatim,hyperref}
\usepackage{ulem}
\usepackage[symbol]{footmisc}
\usepackage{enumitem,longtable}
\usepackage{MnSymbol,wasysym}
\usepackage{subfig}
\RequirePackage{etex}
\usepackage{fullpage,cancel}
\usepackage{float}
\usepackage{wrapfig}
\usepackage[dvipsnames]{xcolor}
\usepackage{graphicx}
\usepackage{multicol}
\usepackage{multirow}
\usepackage{mathrsfs}
\usepackage{longtable}
\usepackage{makecell}
\setcellgapes{4pt}
\usepackage{supertabular}
\usepackage{tabu}
\usepackage{tikz}
\usepackage{tkz-graph}
\definecolor{darkviolet}{rgb}{0.58, 0.0, 0.83}
\tikzstyle{vertex}=[circle, draw, inner sep=0pt, minimum size=4pt]

\tikzstyle{vtx}=[circle, draw, inner sep=0pt, minimum size=8pt]

\usepackage{colortbl}

\def\ZZ{{\mathbb Z}}
\def\QQ{{\mathbb Q}}
\def\RR{{\mathbb R}}

\def\CC{{\mathbb C}}
\usepackage{relsize}
\newenvironment{psmallmatrix}
  {\left(\begin{smallmatrix}}
  {\end{smallmatrix}\right)}
\newcommand\evsout{\bgroup\markoverwith{\textcolor{darkviolet}{\rule[.5ex]{4pt}{.6pt}}}\ULon}
\newcounter{counter} \numberwithin{counter}{section}    
\newtheorem{proposition}[counter]{Proposition}
\newtheorem{theorem}[counter]{Theorem}
\newtheorem{definition}[counter]{Definition}
\let\olddefinition\definition
\renewcommand{\definition}{\olddefinition\normalfont}
\newtheorem{remark}[counter]{Remark}
\newtheorem{corollary}[counter]{Corollary}

\newtheorem{example}[counter]{Example}

\newcommand{\pmattwotwo}[4]{\begin{pmatrix} #1 & #2\\ #3 & #4\end{pmatrix}}

\DeclareMathOperator{\Gal}{Gal}
\DeclareMathOperator{\Ded}{Ded}
\numberwithin{equation}{section}
\numberwithin{figure}{section}

\title{The kernel of newform Dedekind sums}

\author{Evuilynn Nguyen}
\address{Department of Mathematics, Rhodes College, United States}
\email{{\href{mailto:nguet-21@rhodes.edu}{nguet-21@rhodes.edu}}}

\author{Juan J. Ramirez}
\address{Department of Mathematics, University of Houston, United States}
\email{\href{mailto:juan.ramirez789456@gmail.com}{jjramirez8@uh.edu}}

\author{Matthew P. Young}
\address{Department of Mathematics, Texas A\&M University, United States}
\email{\href{mailto:myoung@math.tamu.edu}{myoung@math.tamu.edu}}

\begin{document}
\begin{abstract}
Newform Dedekind sums are a class of crossed homomorphisms that arise from newform Eisenstein series.
We initiate a study of the kernel of these newform Dedekind sums.   Our results can be loosely described as showing that these kernels are neither ``too big" nor ``too small."
We conclude with an observation about the Galois action on Dedekind sums that allows for significant computational efficiency in the numerical calculation of Dedekind sums.
\end{abstract}
\maketitle
\section{Introduction}
\subsection{Background and prior work}
Dedekind sums were first introduced by Dedekind as a way to express the transformation formula satisfied by the Dedekind $\eta$-function.  The study of Dedekind sums have shown up in different areas of mathematics including algebraic number theory, combinatorial geometry, topology, and mathematical physics.  Let $h,k$ be coprime integers with $k \geq 1$. The classical Dedekind sum is defined 
by
\begin{equation}
    s(h,k) = \sum_{n \tiny{\mbox{ mod }} k}B_1\Big(\frac{n}{k}\Big) B_1\Big(\frac{hn}{k}\Big),
    \label{eq:classicaldedsum}
\end{equation}
where $B_1$ denotes the first Bernoulli function defined by
\begin{align}\label{eq:bernoulli}
 B_1(x) &=
  \begin{cases}
   x-\lfloor x \rfloor - \tfrac{1}{2}        & \text{if } x \in \RR \setminus \ZZ \\
   0        & \text{if } x \in \ZZ.
  \end{cases}
\end{align}
For more background on the Dedekind $\eta$-function and the classical Dedekind sum, we refer the reader to \cite{modfunc}.

Many different generalized versions of the Dedekind sum have appeared in the literature. In this paper, we continue the study of newform Dedekind sum examined recently by Stucker, Vennos, and Young in \cite{stuckervennosyoung} and Dillon and Gaston \cite{dillongaston}; see their introductions for a more thorough historical survey of previous work on these types of Dedekind sums.  

We now summarize the construction of newform Dedekind sums and some of their basic properties which follow from properties of an associated Eisenstein series (via a generalized Kronecker limit formula).  To this end, we first discuss the Eisenstein series.
Throughout we let $\chi_1$ and $\chi_2$ be primitive Dirichlet characters modulo $q_1$ and $q_2$, respectively, with $q_1$, $q_2 > 1$, and satisfying $\chi_1 \chi_2(-1) = 1$. Let $\Gamma_0(q_1 q_2)$ denote the congruence subgroup of level $q_1 q_2$.  The newform Eisenstein series $E_{\chi_1, \chi_2}(z,s)$ associated to $\chi_1, \chi_2$ may be defined by the Fourier expansion
\begin{equation}
\label{eq:EisensteinFourierExpansion}
E_{\chi_1, \chi_2}(z,s) = 2 \sqrt{y} \sum_{n \neq 0} \lambda_{\chi_1,\chi_2}(n,s) e(nx) K_{s-\frac12}(2 \pi |n| y),
\end{equation}
where 
\begin{equation}
\lambda_{\chi_1, \chi_2}(n,s) = \chi_2(\text{sgn}(n)) \sum_{ab=|n|} \chi_1(a) \overline{\chi_2}(b) \Big(\frac{b}{a}\Big)^{s-\frac12}.  
\end{equation}
 This Eisenstein series enjoys a host of pleasant properties, such as:
\begin{enumerate}
\item It satisfies the automorphy formula $E_{\chi_1, \chi_2}(\gamma z, s) = \psi(\gamma) E_{\chi_1, \chi_2}(z,s)$, for all $\gamma \in \Gamma_0(q_1 q_2)$, where $\psi = \chi_1 \overline{\chi_2}$, and $\psi(\gamma) = \psi(d)$ for $\gamma = (\begin{smallmatrix} a & b \\ c & d \end{smallmatrix})$.
\item It has analytic continuation to all $s \in \CC$, with no poles.
\item It satisfies $\Delta E_{\chi_1, \chi_2}(z,s) = s(1-s) E_{\chi_1, \chi_2}(z,s)$, where $\Delta$ is the hyperbolic Laplacian.
\item It is an eigenfunction of all the Hecke operators.
\item It is a pseudo-eigenfunction of the Atkin-Lehner operators.
\end{enumerate}
All of these properties may be conveniently found in \cite{YoungEisenstein}, for instance.

The reason for the name ``newform" Eisenstein series is by perfect analogy with the corresponding notion for cusp forms. 

One should be aware that if $q_1 = 1$ or $q_2 = 1$ then there exists an additional constant term in 
\eqref{eq:EisensteinFourierExpansion}, and if $q_1 = q_2=1$ then $E_{\chi_1, \chi_2}(z,s)$ has simple poles at $s=0,1$ only.  Many properties of the classical Dedekind sum \eqref{eq:classicaldedsum} can be deduced from the behavior of the level $1$ Eisenstein series, especially its Laurent expansion around $s=1$.  The presence of the pole at $s=1$ is something of a nuisance, and it is a pleasant fact that all the newform Eisenstein series are analytic at $s=1$.  

Using the evaluation $K_{1/2}(2 \pi y) = 2^{-1} y^{-1/2} e^{-2 \pi y}$ in \eqref{eq:EisensteinFourierExpansion}, we obtain a decomposition
\begin{equation}
\label{eq:fintermsofE}
E_{\chi_1, \chi_2}(z,1) = f_{\chi_1, \chi_2}(z) + \chi_2(-1) \overline{f}_{\overline{\chi_1}, \overline{\chi_2}}(z),
\end{equation}
where
\begin{equation}
\label{eq:fchi1chi2FourierExpansion}
f_{\chi_1, \chi_2}(z) = \sum_{n=1}^{\infty} \frac{\lambda_{\chi_1, \chi_2}(n,1)}{\sqrt{n}} e^{2 \pi i n z}.
\end{equation}
Note that $f_{\chi_1, \chi_2}(z)$ is holomorphic on $\mathbb{H}$, is periodic with period $1$, and vanishes at $i \infty$.  We also observe that the level $1$ analog of $f_{\chi_1, \chi_2}(z)$ is closely related to $\log \eta$.  One may define the newform Dedekind sum $S_{\chi_1, \chi_2}$ as a correction factor to the automorphy of $f_{\chi_1, \chi_2}(z)$, precisely, for $\gamma \in \Gamma_0(q_1 q_2)$, we let
\begin{equation}
\label{eq:DedekindSumDefviaEisensteinSeries}
\frac{\pi i}{\tau(\overline{\chi_1})} S_{\chi_1, \chi_2}(\gamma) = f_{\chi_1, \chi_2}(\gamma z) - \psi(\gamma) f_{\chi_1, \chi_2}(z),
\end{equation}
where $\tau(\chi)$ denotes the standard Gauss sum; the factor $\frac{\pi i}{\tau(\overline{\chi_1})}$ is a normalization to simplify various formulas.  From the automorphy of $E_{\chi_1, \chi_2}$
and \eqref{eq:fintermsofE} we deduce that the right hand side of \eqref{eq:DedekindSumDefviaEisensteinSeries} is 
both holomorphic and anti-holomorphic, and is hence
a constant function of $z$; this explains why the left hand side of \eqref{eq:DedekindSumDefviaEisensteinSeries} only depends on $\gamma$.
For simplicity, we may refer to the newform Dedekind sum simply as the Dedekind sum.  

As observed in \cite[Section 5]{stuckervennosyoung}, an alternative expression for $f_{\chi_1, \chi_2}$ is as an Eichler integral of a weight $2$ holomorphic Eisenstein series associated to the characters $\chi_1, \chi_2$.

The most important basic property of the Dedekind sum $S_{\chi_1,\chi_2}: \Gamma_0(q_1q_2) \rightarrow \CC$ is that it is a crossed homomorphism, which we record with the following:
\begin{theorem}[Crossed homomorphism identity]
For all $\gamma_1, \gamma_2 \in \Gamma_0(q_1 q_2)$, we have
\begin{align}\label{eq:hom}
S_{\chi_1,\chi_2}(\gamma_1\gamma_2) = S_{\chi_1,\chi_2}(\gamma_1)+\psi(\gamma_1)S_{\chi_1,\chi_2}(\gamma_2).
\end{align} 
\end{theorem}
This follows directly from \eqref{eq:DedekindSumDefviaEisensteinSeries} (see \cite[Lemma 2.2]{stuckervennosyoung} for more details).
\begin{remark}
For $\gamma \in \Gamma_1(q_1q_2)$, then $\psi(\gamma) = 1$, so $S_{\chi_1,\chi_2}: \Gamma_1(q_1 q_2) \rightarrow \CC$ is a group homomorphism.
\end{remark} 

Although \eqref{eq:DedekindSumDefviaEisensteinSeries} is useful for establishing many properties of the Dedekind sum, it is not explicit; the Fourier expansion \eqref{eq:fchi1chi2FourierExpansion} is an infinite sum.  Theorem 1.2 from \cite{stuckervennosyoung} evaluates the Dedekind sum in finite terms, as follows.
For $\gamma = \begin{psmallmatrix} a & b\\ c & d  \end{psmallmatrix} \in \Gamma_0(q_1q_2)$ with $c>0$ and $\chi_1\chi_2(-1) = 1,$ we have
\begin{equation}
    \label{eq:dedsum}
    S_{\chi_1,\chi_2}(\gamma) = \sum_{j \bmod c} \sum_{n \bmod q_1} \overline{\chi_2}(j)\overline{\chi_1}(n)B_1\left(\frac{j}{c}\right)B_1\left(\frac{n}{q_1}+\frac{aj}{c}\right).
\end{equation} 
Since (\ref{eq:dedsum}) only depends on the first column of $\gamma$, we will often write $S_{\chi_1,\chi_2}(a,c)$ in place of $S_{\chi_1,\chi_2}(\gamma)$. 

It is easy to see from \eqref{eq:DedekindSumDefviaEisensteinSeries} that for $c =0$ we have $S_{\chi_1, \chi_2}(1,0) = 0$, and that $S_{\chi_1, \chi_2}(-a,-c) = S_{\chi_1, \chi_2}(a,c)$, which takes care of $c <0$.  It is also easy to see that for $c>0$, $S_{\chi_1, \chi_2}(a,c)$ is periodic in $a$ modulo $c$; this is obvious from \eqref{eq:dedsum} but has its origin from the fact that $f_{\chi_1, \chi_2}$ is periodic with period $1$.

The reciprocity formula for the classical Dedekind sum is one of its most interesting and important features.
The following reciprocity formula for $S_{\chi_1,\chi_2}$ is proved in \cite{stuckervennosyoung} via the action of the Fricke involution $\omega = \begin{psmallmatrix}0&-1\\q_1q_2&0 \end{psmallmatrix}$:
\begin{theorem}[Reciprocity formula \cite{stuckervennosyoung}]
\label{thm:reciprocity}
For $\gamma=\begin{psmallmatrix}a&b\\cq_1q_2&d \end{psmallmatrix} \in \Gamma_0(q_1q_2),$ let $\gamma' =\begin{psmallmatrix}d&-c\\-bq_1q_2&a \end{psmallmatrix} \in \Gamma_0(q_1q_2).$ If $\chi_1,\chi_2$ are even, then \begin{align}
    S_{\chi_1,\chi_2}(\gamma) =  S_{\chi_2,\chi_1}(\gamma').
    \label{eq:repeven}
\end{align} If $\chi_1,\chi_2$ are odd, then with  $\tau(\chi)$ denoting the standard Gauss sum,
we have
\begin{align}
    S_{\chi_1,\chi_2}(\gamma) =  -S_{\chi_2,\chi_1}(\gamma')+(1-\psi(\gamma))\left(\frac{\tau(\overline{\chi}_1)\tau(\overline{\chi}_2)}{(\pi i )^2}\right)L(1,\chi_1)L(1,\chi_2).
    \label{eq:repodd}
\end{align}
\end{theorem}

Our main interest in this paper is to understand the structure of the kernels of the Dedekind sums.
To make our objects of interest more precise, we make the following definition. 
\begin{definition}
 Let $\chi_1$ and $\chi_2$ be non-trivial primitive Dirichlet characters modulo $q_1$ and $q_2$, respectively, with $q_1,q_2 > 1.$ Then we denote the kernel associated to $\chi_1,\chi_2$, by
\begin{align*}
    K_{\chi_1,\chi_2} &= \ker(S_{\chi_1, \chi_2}) = \left\{\gamma \in \Gamma_0(q_1q_2): S_{\chi_1,\chi_2}(\gamma) = 0 \right\}.
\end{align*} We let $K_{\chi_1,\chi_2}^1$ denote $K_{\chi_1,\chi_2} \bigcap \Gamma_1(q_1q_2)$.
Moreover, we define
$$ K_{q_1,q_2} = \bigcap_{\substack{\chi_1, \chi_2\\ \chi_1 \chi_2(-1)=1}} K_{\chi_1,\chi_2},
$$ 
where $\chi_i$ runs over primitive characters modulo $q_i$, $i=1,2$.
We similarly let $K_{q_1,q_2}^1= K_{q_1,q_2} \bigcap \Gamma_1(q_1q_2)$.
\label{def:ker}
\end{definition}
It is easy to see that $K_{\chi_1, \chi_2}$ (resp. $K_{\chi_1, \chi_2}^1$) is a subgroup of $\Gamma_0(q_1 q_2)$ (resp. $\Gamma_1(q_1 q_2)$).  Given any group homomorphism, it is a fundamental question to study its kernel.  In this context, there is an additional curiosity which is that
\begin{equation}
\gamma \in K_{\chi_1, \chi_2} \quad \Longleftrightarrow \quad f_{\chi_1, \chi_2}(\gamma z) = \psi(\gamma) f_{\chi_1, \chi_2}(z).
\end{equation}
In words, the elements of $K_{\chi_1, \chi_2}$ are precisely those for which $f_{\chi_1, \chi_2}$ transforms like an automorphic form.  It is well-known that there are no weight $0$ holomorphic modular forms, and the size of $K_{\chi_1, \chi_2} \backslash \Gamma_0(q_1 q_2)$ may, in some loose sense, be interpreted to measure the failure of $f_{\chi_1, \chi_2}$ to be modular.

\begin{remark}
\label{remark:Gamma1andGamma}
One can similarly consider $K_{\chi_1, \chi_2} \bigcap \Gamma(q_1 q_2)$, but since $S_{\chi_1, \chi_2}$ 
depends only on the first column of $\gamma$, this is essentially the same as $K_{\chi_1, \chi_2}^1$. A more precise statement is that $\Gamma_1(q_1 q_2) = \bigcup_{b} \Gamma(q_1 q_2) (\begin{smallmatrix} 1 & b \\ 0 & 1 \end{smallmatrix})$, and $(\begin{smallmatrix} 1 & b \\ 0 & 1 \end{smallmatrix})$ is trivially in the kernel of any Dedekind sum.
\end{remark}

The following theorem of Dillon and Gaston \cite{dillongaston} shows that $S_{\chi_1, \chi_2}$ is non-trivial in a strong sense:
\begin{theorem}[Strong nontriviality \cite{dillongaston}]
\label{thm:DillonGaston} For each $c > 0$ such that $q_1 q_2 | c$, there exists $a \in \ZZ$  so that $S_{\chi_1, \chi_2}(a,c) \neq 0$.
\end{theorem}
\begin{remark} One way to interpret this result of Dillon and Gaston is that it shows that $K_{\chi_1, \chi_2}$ is not ``too big" (keeping account of the size of $c$, the lower-left entry of elements of $\Gamma_0(q_1 q_2)$).  
\end{remark}
\begin{remark}
The Eichler-Shimura isomorphism can be used to show that $S_{\chi_1, \chi_2}$ is non-trivial,  but does not obviously imply the strong non-triviality.
\end{remark}

Our first main result shows that $K_{q_1, q_2}^1$ is not ``too small."
\begin{theorem}[Kernel is strongly nontrivial]
\label{thm:kernelisbig}
For every $c \in \ZZ$, there exists $\gamma  = \begin{psmallmatrix}a&b\\cq_1q_2&d  \end{psmallmatrix} \in \Gamma(q_1q_2)$ such that $\gamma \in K_{q_1,q_2}^1$.
\label{cor:nonempty}
\end{theorem}
Although Theorem \ref{thm:kernelisbig} shows that the kernel of Dedekind sums is not too small, this must be balanced against the following:
\begin{proposition}
\label{prop:indexisinfinite}
The index of $K_{\chi_1, \chi_2}^1$ in $\Gamma_1(q_1 q_2)$ is infinite.
\end{proposition}

Next we discuss some relationships between commutator subgroups and kernels of the newform Dedekind sums. We begin with a general discussion.  If $G$ is a group, we let $[G,G]$ denote its commutator subgroup (i.e., the smallest subgroup of $G$ containing all commutators $xyx^{-1} y^{-1}$ with $x,y \in G$).
It is well-known (and easy to check) that if $\varphi: G \rightarrow H$ is a group homomorphism, with $H$ abelian, then $[G,G] \subseteq \ker(\varphi)$.  We also recall that the abelianization of a group, denoted $G^{\text{ab}}$ is defined by $G^{\text{ab}} = [G,G] \backslash G$.  It is known that the abelianization of $SL_2(\ZZ)$ is $\ZZ/12\ZZ$, which implies that there are no non-trivial group homomorphisms from $SL_2(\ZZ)$ to $\CC$.  Theorem \ref{thm:DillonGaston} is in sharp contrast to the level $1$ case.

One naturally is led to wonder to what extent the commutator subgroups of $\Gamma_0(q_1 q_2)$, $\Gamma_1(q_1 q_2)$, etc. account for the kernels of the Dedekind sums.  Our second main result shows that $K_{q_1, q_2}^1$ is much larger than the commutator subgroup of $\Gamma(q_1 q_2)$ (cf. Remark \ref{remark:Gamma1andGamma}).
\begin{theorem}
\label{thm:commutator}
We have $[\Gamma(q_1 q_2), \Gamma(q_1 q_2)] \subsetneq K_{q_1, q_2}^1$.
\end{theorem}
\begin{remark}
In fact, we show in Proposition \ref{prop:comm} below that $[\Gamma(q_1 q_2), \Gamma(q_1 q_2)] \subseteq \Gamma(q_1^2 q_2^2)$.  In contrast, Theorem 
\ref{thm:kernelisbig} produces elements that are clearly not in $\Gamma(q_1^2 q_2^2)$ (indeed, there is no restriction on the lower-left entry besides divisibility by $q_1 q_2$).  This explains why we stated that $K_{q_1, q_2}^1$ is \emph{much} larger than the commutator subgroup.
\end{remark}

Our final main observation is that there exists a natural Galois action on the Dedekind sums, which can easily be read off from \eqref{eq:dedsum}.  This is discussed in Section \ref{section:Galois}.

\section{Numerical data and proof of Theorem \ref{thm:kernelisbig}}
We begin this section by presenting some numerical calculations of $K_{q_1,q_2}$. We let $(a,c)$ represent the left column of $\gamma \in \Gamma_0(q_1q_2)$. By (3), it can be shown that $S_{\chi_1,\chi_2}(a,c) = S_{\chi_1,\chi_2}(b,c)$ where $a \equiv b \bmod c$. Therefore, we only need to examine the pairs $(a,c)$ such that $a \in \{1,...,c-1\}.$ Using SageMath \cite{sage}, for all primes $3 \leq q_1,q_2 \leq 11$, we computed the elements of $K_{\chi_1,\chi_2}, K_{\chi_1,\chi_2}^1,K_{q_1,q_2}$, and $K_{q_1,q_2}^1$ with $1 \leq c \leq 10q_1q_2$, directly using the finite sum formula \eqref{eq:dedsum}. Consider the example in Figure \ref{subfig:kint55} where $q_1=q_2=5$ in which we display the elements of $K_{5,5}$ for $1\leq c\leq 250.$
\begin{figure}[H]
    \centering
    \subfloat[$K_{5,5}$]{{\includegraphics[scale=.4]{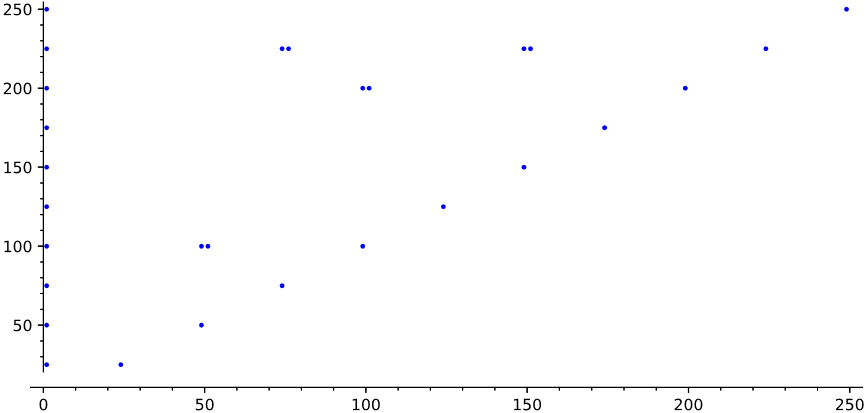}}
    \label{subfig:kint55}}
    \hfill
    \subfloat[$K_{7,3}$]{{\includegraphics[scale=.4]{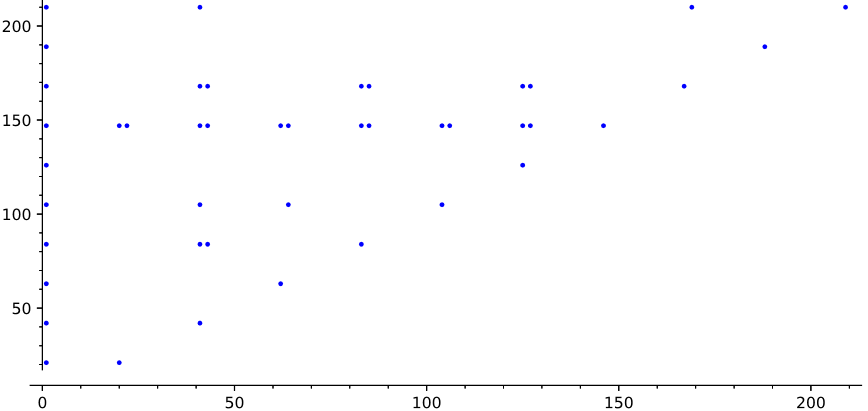}}
    \label{subfig:kint73}}\hfill
    \caption{$K_{q_1,q_2}$ for $1\leq c\leq 10q_1q_2$}
    \label{fig:kerint}
\end{figure}
From Figure \ref{subfig:kint55}, Figure \ref{subfig:kint73}, and other similar graphs, we found the vertical line formed when $a=1$ to consistently appear.  We prove this in Corollary \ref{cor:a1}.  We also found other lines corresponding to similar patterns shown in the following propositions. 
\begin{proposition}
Let $\chi_1$ and $\chi_2$ be non-trivial primitive Dirichlet characters modulo $q_1$ and $q_2$, respectively, with $q_1,q_2 > 1$. Then $S_{\chi_1,\chi_2}(1,q_1q_2) = 0$.
\label{prop:a1value}
\end{proposition}
\begin{proof}
We take $\gamma = (\begin{smallmatrix} 1 & 0 \\ q_1 q_2 & 1 \end{smallmatrix})$ in Theorem \ref{thm:reciprocity}, so $\gamma' = (\begin{smallmatrix} 1 & -1 \\ 0 & 1 \end{smallmatrix})$.  It is easy to see that $S_{\chi_1, \chi_2}(\gamma') = 0$, and that $1-\psi(\gamma) = 0$.
Therefore, Proposition \ref{prop:a1value} follows from the reciprocity formula.
\end{proof}
In \cite[Proposition 2.3]{dillongaston}, Dillon and Gaston pointed out that for $c\geq 1$ and $q_1q_2 | c$,
\begin{align}
    S_{\chi_1,\chi_2}(-a,c) &= -\chi_2(-1)S_{\chi_1,\chi_2}(a,c).
    \label{eq:a_c-a}
\end{align} Letting $c = q_1q_2$ and $a = 1$ in \eqref{eq:a_c-a}, we can use Proposition \ref{prop:a1value}, to conclude that $S_{\chi_1,\chi_2}(q_1q_2-1,q_1q_2)=0$. Moreover, using \eqref{eq:a_c-a}, one can easily observe that for $a \pmod c$, if $S_{\chi_1,\chi_2}(a,c) = 0$ then $S_{\chi_1,\chi_2}(c-a,c) = 0$. This symmetry between the pairs $(a,c)$ and $(c-a,c)$ can be seen in Figure \ref{fig:kerint} above. 

Dillon and Gaston \cite[Proposition 2.4]{dillongaston} also observed that
\begin{equation}
\label{eq:abara}
S_{\chi_1, \chi_2}(\overline{a}, c) = \chi_2(-1) \psi(a) S_{\chi_1, \chi_2}(a,c),
\end{equation}
where $a \overline{a} \equiv 1 \mod{c}$.
\begin{proposition}
\label{prop:squaresinkernel}
Suppose that $a \in \ZZ$ is so that $a^2 + 1 \equiv 0 \mod{c}$, and that $\psi(a) = +1$.  Then $S_{\chi_1, \chi_2}(a,c) = 0$.  Similarly, if $a^2 \equiv 1 \mod{c}$, and $\psi(a) = - \chi_2(-1)$, then $S_{\chi_1, \chi_2}(a,c) = 0$.
\end{proposition}
\begin{proof}
First consider the case where $a^2 + 1 \equiv 0 \mod{c}$, i.e., $a \equiv - \overline{a} \mod{c}$. Combining \eqref{eq:a_c-a} and \eqref{eq:abara}, we deduce 
$$S_{\chi_1, \chi_2}(a,c) = S_{\chi_1, \chi_2}(-\overline{a}, c)
= - \chi_2(-1) S_{\chi_1, \chi_2}(\overline{a},c) 
= - \psi(a) S_{\chi_1, \chi_2}(a,c). 
$$
The case with $a^2 \equiv 1 \mod{c}$ is even simpler, using only \eqref{eq:abara}.
\end{proof}
\begin{remark}
Proposition \ref{prop:squaresinkernel} is a simple generalization of a well-known result for the classical Dedekind sum; see \cite[Theorem 3.6(c)]{modfunc}.
\end{remark}

\begin{example}
From Figure \ref{subfig:kint55}, we see the first occurrence of $(a,c)$ in $K_{5,5}$ with $a \not \equiv \pm 1 \mod{c}$ is with $c=100$, and $a=49,51$.  We now explain how to prove (without a computer) that these points are in $K_{\chi_1,\chi_2}$ for $\chi_1$ and $\chi_2$ odd.
Note that $51^2 \equiv 49^2 \equiv 1 \mod{100}$.  Since $51 \equiv 1 \mod{25}$, this means 
$\psi(50 \pm 1) = 1$.
If $\chi_1$ and $\chi_2$ are odd, then Proposition \ref{prop:squaresinkernel} implies $S_{\chi_1, \chi_2}(50 \pm 1, 100) = 0$, which is the desired claim.
This argument does not work for $\chi_1$ and $\chi_2$ even.   
We leave it as a problem for the interested reader to cover this remaining case (one may note that the only even primitive character modulo $5$ is the Legendre symbol, so there is only one case left). 
\end{example}

\begin{proposition}
If $S_{\chi_1,\chi_2}(1+ndq_1q_2,d^2q_1q_2)=0$ for some $n,d \in \ZZ$,  then \\ $S_{\chi_1,\chi_2}(1+kndq_1q_2,kd^2q_1q_2)=0$ for all $k \in \ZZ$. 
\label{prop:mulsq}
\end{proposition} 
\begin{remark} 
Proposition \ref{prop:mulsq} can be used to explain some linear patterns visible among the points in Figure \ref{fig:kerint}.  For instance, we have $a=51$, $c=100$ in $K_{5,5}$ visible in Figure \ref{subfig:kint55}, which corresponds to $d=2$, $n=1$, in Proposition \ref{prop:mulsq}.  The point $a=101$, $c=200$ then corresponds to $k=2$ in Proposition \ref{prop:mulsq}.
\end{remark}
\begin{proof}
We prove this by considering matrices of the form $$\gamma  = \pmattwotwo{1+ndq_1q_2}{-n^2q_1q_2}{d^2q_1q_2}{1-ndq_1q_2}
=
I + QA,
\quad \text{where} 
\quad
A = \pmattwotwo{nd}{-n^2}{d^2}{-nd},
\quad Q = q_1 q_2.
$$ 
Note that $A^2 = 0$.  Using this, one can easily show that $\gamma^k = I+kQA$, for any $k \in \mathbb{Z}$.  If $\gamma \in K_{\chi_1, \chi_2}$, then so is $\gamma^k$, which translates to the desired statement.
\end{proof}

\begin{corollary}
We have that
$S_{\chi_1,\chi_2}(1,kq_1q_2)=0$ for all $k \in \ZZ$.
\label{cor:a1}
\end{corollary}
\begin{proof}
Apply Propositions \ref{prop:a1value} and \ref{prop:mulsq}, with $n=0$ and $d =1$.
\end{proof}
 The points $(1,c)$ create a vertical line, depicted in Figure \ref{fig:kerint} above. Similarly, one can use \eqref{eq:a_c-a} to conclude that the points $(c-1,c)$ create a line of slope $1$, also depicted in Figure \ref{fig:kerint}. 

Now we prove Theorem \ref{thm:kernelisbig}. 
By Corollary \ref{cor:a1}, we have that $\begin{psmallmatrix} 1&0\\cq_1q_2&1 \end{psmallmatrix} \in K_{\chi_1,\chi_2}^1$, for all choices of $\chi_1$, $\chi_2$, so the result follows immediately.
\qed
\begin{remark}
Figure \ref{fig:kerint} indicates that there exist examples of $q_1, q_2$ and $c$ for which the \emph{only} element $(a,c) \in K_{q_1, q_2}^{1}$ with $0 < a <  c$ is the point exhibited in Corollary \ref{cor:a1}.  In this sense, Corollary \ref{cor:a1} is sharp.
\end{remark}    

Now we briefly prove Proposition \ref{prop:indexisinfinite}.  By a group isomorphism theorem, we have
\begin{equation}
K_{\chi_1, \chi_2}^1 \backslash \Gamma_1(q_1 q_2) \simeq S_{\chi_1, \chi_2}(\Gamma_1(q_1 q_2)) \subseteq \CC.
\end{equation}
Since every nonzero element of $\mathbb{C}$ has infinite order, every nontrivial element of the quotient group $K_{\chi_1, \chi_2}^1 \backslash \Gamma_1(q_1 q_2)$ also has infinite order.  
In particular, this proves Proposition \ref{prop:indexisinfinite}.  This argument shows that if $\gamma \in \Gamma_1(q_1 q_2)$ is such that $\gamma^k \in K_{\chi_1, \chi_2}^1$ for some integer $k \geq 1$, then $\gamma \in K_{\chi_1, \chi_2}^1$.

\section{The commutator subgroup: Proof of Theorem \ref{thm:commutator}}

For any prime $p$, let $\Gamma(n;p)$ denote the principal congruence subgroup of $SL_n(\ZZ)$ of level $p$. In \cite{commutator}, Lee and Szczarba show that $[\Gamma(n;p),\Gamma(n;p)] = \Gamma(n;p^2)$ for $n \geq 3$ and all primes $p$. In the following proposition, we adapt the proof of Lee and Szczarba to show a one-sided containment of the commutator subgroup for $n=2$ and $p$ not necessarily prime.
\begin{proposition}
For any $Q \geq 1$, then $[\Gamma(Q),\Gamma(Q)] \subseteq \Gamma(Q^2)$.
\label{prop:comm}
\end{proposition}
\begin{proof}
Define the map $\varphi: \Gamma(Q) \rightarrow M_{2\times2}(\ZZ/Q\ZZ)$ by $\varphi(A) = \frac{A - I}{Q} \pmod{Q}$. It is not difficult to show $\varphi$ is a group homomorphism. Since $M_{2\times2}(\ZZ/Q\ZZ)$ is abelian, then $[\Gamma(Q), \Gamma(Q)] \subseteq \ker(\varphi)$. We can see that $\ker(\varphi) = \Gamma(Q^2)$ by the definition of $\varphi$.
\end{proof}
\begin{remark}
The proof of Lee and Szczarba may be easily adapted to additionally show that the image of $\varphi$ is the subset of $M_{2\times 2}(\ZZ/Q\ZZ)$ of trace $\equiv 0 \mod{Q}$.
\end{remark}


Proposition \ref{prop:comm} shows that $[\Gamma(q_1q_2),\Gamma(q_1q_2)] \subseteq \Gamma(q^2_1q^2_2)$. 
From Proposition \ref{prop:mulsq} and Corollary \ref{cor:a1}, we see that the inclusion is strict, i.e. $\Gamma(q^2_1q^2_2) \subsetneq K^1_{q_1,q_2}$.  This completes the proof of Theorem \ref{thm:commutator}.

\section{The Galois action}
\label{section:Galois}
We now study the kernels further by comparing $K_{\chi_1, \chi_2}$ for different choices of $\chi_1, \chi_2$, with a specified choice of $q_1,q_2$. Let  $\zeta_n = e^{2\pi i/n}$. Figure \ref{fig:d511} depicts the elements of $K_{\chi_1,\chi_2}^1$ with $0 < c \leq 1100$ for the Dedekind sum associated to $\chi_1 \mod 5$, the character mapping $2 \mapsto i$, and $\chi_{2} \mod 11$, which maps $2 \mapsto \zeta_{10}$. However, we found that Figure \ref{fig:d511} also represents the kernel for other characters, 
such as the pair $\chi'_1 \bmod 5: 2 \mapsto - i$ and $\chi_{2}' \bmod 11: 2 \mapsto \zeta_{10}^{3}.$
\begin{figure}[H]
    \centering
    \includegraphics[scale=0.4]{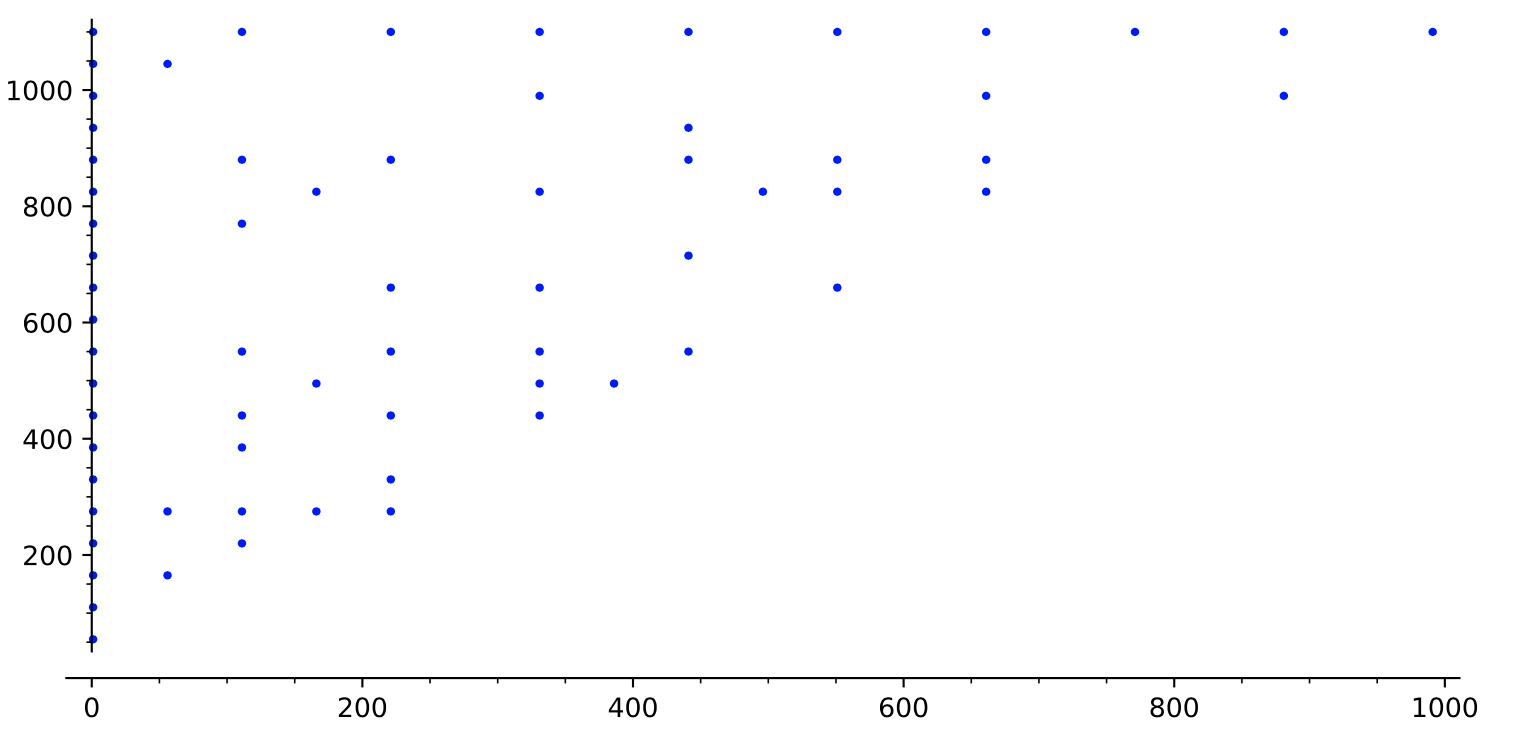}
    \caption{ $K_{\chi_1,\chi_2}^1$ for $\chi_1 \bmod 5$ and $\chi_2 \bmod 11$}
    \label{fig:d511}
\end{figure}
This reoccurring pattern of identical kernels was found for other conductors $q_1,q_2$ as well. To explain this pattern we use Galois theory.
For characters $\chi_1,\chi_2 \mod q_1$ and $q_2$ respectively, we let $F = \QQ\left(\zeta_{\phi(q_1)}, \zeta_{\phi(q_2)}\right) = \QQ(\zeta_{[\phi(q_1), \phi(q_2)]})$ be the field extension over $\QQ$ generated by the $\phi(q_1)$-th and $\phi(q_2)$-th  primitive roots of unity. Let 
\begin{align}
    \Ded(q_1,q_2) = \{S_{\chi_1,\chi_2} ~|~ \chi_1\chi_2(-1)=1 \text{ and } \chi_i \text{ primitive modulo } q_i, i=1,2 \}.
    \label{eq:dedset}
\end{align}

For $\sigma \in \Gal(F/\QQ)$, and $\chi$ a Dirichlet character taking values in $F$, let $\chi^\sigma$ denote the character defined by $n \rightarrow \sigma(\chi(n))$. By the definition of the Dedekind sum in \eqref{eq:dedsum}, we see that $S_{\chi_1,\chi_2}(\gamma)$ lies in $F$, for all $\gamma \in \Gamma_0(q_1 q_2)$, since the values taken by the Bernoulli function $B_1$ in \eqref{eq:dedsum} are rational.

\begin{proposition}
Let $F=\QQ\left(\zeta_{\phi(q_1)}, \zeta_{\phi(q_2)}\right)$.
Then there exists a natural group action of $\Gal(F/\QQ)$ on $\Ded(q_1,q_2)$ which we denote as $S_{\chi_1,\chi_2}^\sigma$ for $\sigma \in \Gal(F/\QQ)$. 
\label{prop:groupact}
\end{proposition}

\begin{proof}
We define the natural Galois action by $S_{\chi_1,\chi_2}^\sigma(\gamma) := \sigma(S_{\chi_1, \chi_2}(\gamma))$, which by the definition \eqref{eq:dedsum} equals $S_{\chi_1^{\sigma}, \chi_2^{\sigma}}(\gamma)$, for any $\gamma \in \Gamma_0(q_1 q_2)$.  From this definition it is easy to see that this is a group action.
\end{proof}

\begin{remark}
If $k\in \ZZ$ is coprime to $\phi(q_1)\phi(q_2),$ the mapping $\omega \rightarrow \omega^k$, where $\omega$ is a primitive root of unity in $F$, is an automorphism of $F$. In fact, all automorphisms in the $\Gal(F/\QQ)$ can be formed this way.
\end{remark}
%

\begin{corollary}
If two Dedekind sums, $S_{\chi_1,\chi_2}$ and $S_{\chi'_1,\chi'_2},$ are in the same orbit of $\Ded(q_1,q_2)$ under the action of $\Gal(F/\QQ)$, then $S_{\chi_1,\chi_2}$ and $S_{\chi'_1,\chi'_2}$ have the same kernel. 
\label{cor:galois}
\end{corollary}

\begin{remark}
Corollary \ref{cor:galois} implies that when studying the kernel of Dedekinds sums associated to specified $q_1,q_2$, we only need to examine a representative for each orbit, which leads to a significant efficiency in computation. 
\end{remark}

\begin{proof}
The statement that $S_{\chi_1, \chi_2}$ and $S_{\chi_1', \chi_2'}$ lie in the same orbit simply means that there exists $\sigma \in \Gal(F/\QQ)$ so that $\chi_1' = \chi_1^{\sigma}$ and $\chi_2' = \chi_2^{\sigma}$.  It is clear that if $\gamma \in K_{\chi_1, \chi_2}$, then $\gamma \in K_{\chi_1^{\sigma}, \chi_2^{\sigma}}$ also.
\end{proof}

\begin{example}
Considering Figure \ref{fig:d511}, letting $k = 3$ we see that 
\begin{itemize}
    \item $q_1=5: \chi_1(2)^3 = i^3 = -i = \chi'_1(2)$ 
    \item $q_2=11: \chi_2(2)^3 = \zeta_{10}^3  = \chi'_2(2)$.
\end{itemize}
It then follows that the Dedekind sums associated to the two pairs of characters are in the same orbit and subsequently have the same kernel, corroborating our corollary.
\end{example}

%

\begin{section}{Acknowledgement}
This work was conducted in summer 2020 during an REU conducted at Texas A\&M University. The authors thank the Department of Mathematics at Texas A\&M and the NSF (DMS-1757872) for supporting the REU. In addition, this material is based upon work supported by the National Science Foundation under agreement No. DMS-170222 (M.Y.).  Any opinions, findings and conclusions or recommendations expressed in this material are those of the authors and do not necessarily reflect the views of the National Science Foundation. The authors would like to thank Agniva Dasgupta (Texas A\&M University) for his immense help and feedback throughout.
\end{section}


\begin{thebibliography}{20}


\bibitem[A]{modfunc} T. Apostol, \textit{Modular Functions and Dirichlet Series in Number Theory.} Graduate Text in Mathematics, 41. Springer-Verlag, New York, 1990.



\bibitem[DG]{dillongaston} T. Dillon, S. Gaston, \textit{An average of generalized Dedekind sums.}  J. Number Theory 212 (2020), 323--338. 


\bibitem[LS]{commutator} R. Lee, R. H. Szczarba, \textit{On the Homology and Cohomology of congruence subgroups.}  Invent. Math. 33 (1976), no. 1, 15--53

\bibitem[SAGE]{sage} SageMath, the Sage Mathematics Software System (Version
  9.0), The Sage Developers, 2020, {\tt https://www.sagemath.org}.

\bibitem[SVY]{stuckervennosyoung} T. Stucker, A. Vennos, M. P. Young, \textit{Dedekind sums arising from newform Eisenstein series.} preprint, arXiv:1907.01524, 2019.

\bibitem[Y]{YoungEisenstein} M. P. Young, \textit{Explicit calculations with Eisenstein series},
J. Number Theory 199 (2019), 1--48.
\end{thebibliography}
\end{document}